\documentclass[12pt,reqno]{amsart}
\usepackage{amsmath,amsthm,amssymb,amsfonts,amscd}
\usepackage{mathrsfs}
\usepackage{bbm}
\usepackage{bbding}
\usepackage[backref=page]{hyperref}
\usepackage{geometry}
\usepackage{color}
\usepackage{xcolor}

\usepackage{picture,epic}
\usepackage{tikz}

\numberwithin{equation}{section}

\theoremstyle{plain}
\newtheorem{theorem}{Theorem}[section]
\newtheorem{lemma}[theorem]{Lemma}
\newtheorem{corollary}[theorem]{Corollary}

\theoremstyle{definition}

\theoremstyle{remark}
\newtheorem{remark}[theorem]{Remark}

\renewcommand{\Re}{\operatorname{Re}}
\renewcommand{\Im}{\operatorname{Im}}
\newcommand{\vol}{\operatorname{vol}}

\newcommand{\supp}{\operatorname{supp}}

\newcommand{\sym}{\operatorname{sym}}

\newcommand{\GL}{\operatorname{GL}}
\newcommand{\SL}{\operatorname{SL}}

\renewcommand{\mod}{\operatorname{mod}\; }

\newcommand{\dd}{\mathrm{d}}

\makeatletter
\def\@tocline#1#2#3#4#5#6#7{\relax
  \ifnum #1>\c@tocdepth 
  \else
    \par \addpenalty\@secpenalty\addvspace{#2}%
    \begingroup \hyphenpenalty\@M
    \@ifempty{#4}{%
      \@tempdima\csname r@tocindent\number#1\endcsname\relax
    }{%
      \@tempdima#4\relax
    }%
    \parindent\z@ \leftskip#3\relax \advance\leftskip\@tempdima\relax
    \rightskip\@pnumwidth plus4em \parfillskip-\@pnumwidth
    #5\leavevmode\hskip-\@tempdima
      \ifcase #1
       \or\or \hskip 1em \or \hskip 2em \else \hskip 3em \fi%
      #6\nobreak\relax
    \hfill\hbox to\@pnumwidth{\@tocpagenum{#7}}\par
    \nobreak
    \endgroup
  \fi}
\makeatother

\begin{document}

\title
{Mixed moments of Hecke eigenforms and $L$-functions}

\author{Bingrong Huang}

\address{Data Science Institute \& State Key Laboratory of Cryptography and Digital Economy Security \\ Shandong University \\ Jinan \\ Shandong 250100 \\China} 

\email{brhuang@sdu.edu.cn}

\dedicatory{Dedicated to Ze\'ev Rudnick on the occasion of his 64th birthday.} 

\date{\today}

\begin{abstract}
  In this paper, we establish estimates for the expectation and variance of the mixed $(2,2)$-moment of two Hecke eigenforms of distinct weights. Our results yield applications to triple product $L$-functions. 
  The proofs are based on moments of $L$-functions.
\end{abstract}



\thanks{This work was supported by  the National Key R\&D Program of China (No. 2021YFA1000700) and 
the Scientific Research Innovation Capability Support Project for Young Faculty (No. SRICSPYF-ZY2025158).}

\maketitle

\section{Introduction} \label{sec:Intr}

The study of the value distribution of automorphic forms is a central problem in analytic number theory and arithmetic quantum chaos.
Let $\mathbb{H}=\{z=x+iy:x\in\mathbb{R},\ y>0\}$ denote the upper half plane, and let $\dd\mu z= \dd x\dd y/y^2$ be the hyperbolic measure.
Let  $\Gamma=\SL_2(\mathbb{Z})$ be the modular group.
Let $k\geq12$ be an even integer,
and let $H_k$ be a Hecke basis for the space $S_k$ of all holomorphic cusp forms of weight $k$ for $\Gamma$.
For $f\in H_k$, we normalize it so that 
$\langle f,f\rangle_k := \int_{\Gamma\backslash \mathbb{H}}  y^k |f(z)|^2 \dd\mu z
=\vol(\Gamma\backslash \mathbb{H})=\pi/3$.

In the $L^2$-setting, we have the holomorphic analog of
the quantum unique ergodicity (hQUE) conjecture of Rudnick and Sarnak \cite{RS}.
In 2010, Holowinsky and Soundararajan \cite{HS} proved hQUE,
 confirming the equidistribution of the mass of $f$.
Specifically, they proved that
\[
  \frac{1}{\vol(\Omega)} \int_{\Omega} y^{k} |f(z)|^{2}  \dd \mu z
    = 1 + o(1),
    \quad \textrm{as $k\rightarrow \infty$,}
\]
for any fixed compact domain $\Omega$ of $\Gamma\backslash\mathbb{H}$ with hyperbolic measure zero boundary $\partial\Omega$.

In 2013, Blomer, Khan, and Young \cite{BlomerKhanYoung2013distribution} studied the $L^4$-norm of $f$, proving that
\begin{equation}\label{eqn:L^4}
  \int_{\Gamma\backslash \mathbb{H}} y^{2k}|f(z)|^{4}  \dd \mu z
  = O( k^{1/3+\varepsilon}),
\end{equation}
for any $\varepsilon>0$.
They conjectured that 
\[
  \frac{1}{\vol(\Gamma\backslash \mathbb{H})}
  \int_{\Gamma\backslash \mathbb{H}} y^{2k}|f(z)|^{4} \dd \mu z=2+o(1), 
  \quad \text{as $k\rightarrow\infty$.}
\]
Assuming the generalized Riemann Hypothesis (GRH), Zenz \cite{Zenz} recently proved
\[
  \int_{\Gamma\backslash \mathbb{H}} y^{2k}|f(z)|^{4}  \dd \mu z = O(1).
\]
Blomer, Khan, and Young \cite{BlomerKhanYoung2013distribution} also showed that for $p>6$, we have
  \[  \int_{\Gamma\backslash \mathbb{H}} |y^{k/2}f(z)|^p \dd\mu z   \gg k^{p/4-3/2-\varepsilon}, \]
which extended the sup-norm result of Xia \cite{Xia}, namely $\max_{z\in\mathbb{H}} |y^{k/2}f(z)| = k^{1/4+o(1)}$.

In general, it is natural to conjecture \cite{Huang2024joint} that
for any $a\in \mathbb{N}$, we have
  \[
    \frac{1}{\vol(\Omega)} \int_{\Omega} y^{ak} |f(z)|^{2a}  \dd \mu z = a! + o(1),
    \quad  \textrm{  as $k\rightarrow \infty$,}
  \]
for any fixed compact domain $\Omega \subset \Gamma\backslash\mathbb{H}$ as above. 
See, for instance, \cite[\S4.4]{HLWY} for the expected value of the $L^p$-norms of random cusp forms.  

\medskip

Recently, we studied the joint distribution of Hecke eigenforms in the large-weight limit.
Let $f\in H_k$ and $g\in H_\ell$ and assume $\langle f,g\rangle=0$ if $k=\ell$.
Denote $F_k(z):=y^{k/2}f(z)$ and $G_\ell(z):=y^{\ell/2}g(z)$.
We know that $|F_k(z)|$ and $|G_\ell(z)|$ are $\Gamma$ invariant.
We will focus on the joint mass distribution of $|F_k(z)|$ and $|G_\ell(z)|$, 
especially the mixed $(2,2)$-moment
\[
  \langle |F_k|^2, |G_\ell|^2 \rangle
  := \int_{\Gamma\backslash\mathbb{H}}    |F_k(z)|^2 |G_\ell(z)|^2 \dd \mu z .
\]
In \cite{Huang2024joint}, we  conjectured that orthogonal Hecke eigenforms are statistically independent.
In particular, we expected that
\begin{equation}\label{eqn:mixed22}
  \frac{1}{\vol(\Gamma\backslash\mathbb{H})}  \langle |F_k|^2, |G_\ell|^2 \rangle = 1 +o(1),
\end{equation}
as $\max(k,\ell)\rightarrow\infty$.
We established \eqref{eqn:mixed22} under the assumptions of the generalized Riemann Hypothesis (GRH) and the generalized Ramanujan conjecture (GRC).

\medskip

In the present paper, we seek unconditional results in this direction.
We consider certain expectation and variance of the mixed $(2,2)$-moment of two Hecke eigenforms of distinct  weights. 
As a consequence, we derive asymptotic formulas for a first moment of triple product $L$-functions, leading to new nonvanishing results for these $L$-functions.

\subsection{The variance}


We  prove that the variance of the mixed $(2,2)$-moment of two Hecke eigenforms vanishes asymptotically.
Our main result on the variance is the following theorem.

\begin{theorem}\label{thm:var}
  Let $K\geq12$ be sufficiently large and  $\ell\geq12$  an even integer.   Assume $\ell \leq K^{\delta_2-\varepsilon}$, with $\delta_2=3/4$. Then for every $g\in H_\ell$,
  \[
    \frac{1}{K^2} \sum_{K<k\leq2K}
    \sum_{f\in H_k}  \Big|\frac{1}{\vol(\Gamma\backslash\mathbb{H})}\langle |F_k|^2, |G_\ell|^2 \rangle-1 \Big|^2
    \ll \ell^{4/3} K^{-1+\varepsilon}.
  \]
\end{theorem}

\begin{remark}
  The theorem can be proved for some small $\delta_2>0$ without much difficulty. 
  The admissibility of $\delta_2=3/4$ relies on the $L^4$-norm bound \eqref{eqn:L^4} for $g\in H_\ell$; see Theorem \ref{thm:mixedmoment} below.
\end{remark}

For even integer $k\geq12$, we have $|H_k|=k/12+O(1)$.
Theorem \ref{thm:var} implies that \eqref{eqn:mixed22} holds for all but $O(K^{2-\varepsilon})$ forms $f\in \cup_{K<k\leq 2K}H_k$, provided $g\in H_\ell$ with $\ell \leq K^{\delta_2-\varepsilon}$.
By the Cauchy--Schwarz inequality, 
we deduce the following asymptotic formula for the expectation. 
\begin{corollary}
  Let $\ell\geq12$ be an even integer. Let $K\geq12$ be sufficiently large.  Assume $\ell \leq K^{\delta_1-\varepsilon}$, with $\delta_1=3/4$. Then for every $g\in H_\ell$,
  \[
    \frac{2}{K}\sum_{\substack{K<k\leq2K \\ k\equiv 0 \bmod 2}} \frac{1}{|H_k|}\sum_{f\in H_k}
     \frac{1}{\vol(\Gamma\backslash\mathbb{H})} \langle |F_k|^2, |G_\ell|^2 \rangle
    = 1 +  O(\ell^{2/3} K^{-1/2+\varepsilon}).
  \]
\end{corollary}

\begin{remark}
  It may be simpler to study the expectation directly, and one might hope to establish an asymptotic formula for some $\delta_1>\delta_2$.
  Recall that Khan \cite{Khan2014} proved the expectation for the $L^4$-norm is $2$ as conjectured by Blomer, Khan, and Young. 
\end{remark}

\subsection{An application to the triple product $L$-functions}

Let $f\in H_k$ with  the $n$-th Hecke eigenvalue $\lambda_f(n)$.
Denote by $\alpha_{f}(p)$, $\alpha_{f}(p)^{-1}$  the Satake parameters  of $f$ at a prime $p$, 
so that $\lambda_f(p)=\alpha_f(p)+\alpha_f(p)^{-1}$.
The symmetric square $L$-function of $f$ is defined by
\[
  L(s,\sym^2 f) = \zeta(2s) \sum_{n\geq1} \frac{\lambda_f(n^2)}{n^s}, \quad \Re(s)>1,
\]
which admits analytic continuation to the entire complex plane, 
and satisfies $L(1,\sym^2f)\neq0$.

Let $f\in H_k$, $g\in H_\ell$ and $h\in H_{k+\ell}$.
The triple product $L$-function of $f,g,h$ is defined  by
\[
  L(s,f\times g\times h) = \prod_{p} \prod_{a=\pm1}\prod_{b=\pm1}\prod_{c=\pm1} \left(1-\frac{\alpha_{f}(p)^{a}\alpha_{g}(p)^{b}\alpha_{h}(p)^{c}}{p^s}\right)^{-1}, \quad \Re(s)>1.
\]
It admits an analytic continuation to the entire complex plane.
By Watson's formula we know $|\langle fg , h\rangle_{k+\ell}|^2$ is related to  the triple product $L$-values $L(1/2,f\times g\times h)$. 
Note that $\langle |F_k|^2, |G_\ell|^2 \rangle  = \langle fg,fg\rangle_{k+\ell} $. 
By Parseval's identity, we have 
\begin{equation}\label{eqn:<f2g2>=L-values}
  \frac{1}{\vol(\Gamma\backslash\mathbb{H})} \langle |F_k|^2, |G_\ell|^2 \rangle  
  =  \frac{2\pi^2}{k+\ell-1} \sum_{h\in H_{k+\ell}} \frac{
    L(1/2,f\times g\times h)}{L(1,\sym^2 h)} 
    \frac{\zeta(2)}{2 L(1,\sym^2 f) L(1,\sym^2 g)}.
\end{equation}
Unconditionally, Blomer, Khan, and Young \cite[Corollary 1.5]{BlomerKhanYoung2013distribution} proved the following nontrivial upper bound for \eqref{eqn:<f2g2>=L-values}
\[
  O((k\ell)^{1/6+\varepsilon}).
\]
Under GRH, by a similar argument as in Zenz \cite{Zenz}, one may prove a sharp upper bound
\[
  O(L(1,\sym^2 f) L(1,\sym^2 g)).
\]

As a consequence  of Theorem \ref{thm:var} and the above discussion, we have the following result for the first moment of $L(1/2,f\times g\times h)$.

\begin{corollary}\label{cor:moment8}
  Let $K\geq12$ be sufficiently large. Let $\ell\geq12$ be an even integer, and   $g\in H_\ell$ a Hecke eigenform. Assume $\ell \leq K^{3/4-\varepsilon}$.
   Then for almost all even integers $k\in(K,2K]$ and almost all $f\in H_k$, we have
  \begin{equation*}
  \frac{2\pi^2}{k+\ell-1} \sum_{h\in H_{k+\ell}} \frac{
    L(1/2,f\times g\times h)}{L(1,\sym^2 h)}
    = 2\frac{L(1,\sym^2 f) L(1,\sym^2 g)}{\zeta(2)} + O(K^{-\varepsilon}).
  \end{equation*}
\end{corollary}


As an immediate corollary, we have the following nonvanishing result for the triple product central $L$-values.

\begin{corollary}\label{cor:nonvanishing}
  Let $K\geq12$ be sufficiently large. Let $\ell\geq12$ be an even integer, and   $g\in H_\ell$ a Hecke eigenform. Assume $\ell \leq K^{3/4-\varepsilon}$.
  Then for almost all even integer $k\in(K,2K]$ and almost all $f\in H_k$, there exists some $h\in H_{k+\ell}$ such that $L(1/2,f\times g\times h)\neq 0$.
\end{corollary}

\subsection{Moments of $L$-functions}

To prove Theorem \ref{thm:var}, we will use various moments of $L$-functions. 
The symmetric square lift $\sym^2 f$  of $f\in H_k$ is a $\GL_3$ automorphic form.
The $(m,n)$-th Fourier coefficient $A(m,n)$ of $\sym^2 f$ is given by
\begin{equation}\label{eqn:A(m,n)}
  A(m,n) = \sum_{d\mid (m,n)} \mu(d) A(m/d,1)A(1,n/d),
\end{equation}
and
\[
  A(n,1)=A(1,n)= \sum_{a^2b=n}\lambda_f(b^2).
\]
Here $\mu$ is the M\"obius function.
Let $\phi$ be a Hecke--Maass cusp form for $\Gamma$, with the $n$-th Hecke eigenvalue $\lambda_\phi(n)$ and the spectral parameter $t_\phi$.
The Rankin--Selberg $L$-function for $\sym^2 f$ and $\phi$ is defined by
\begin{equation}\label{eqn:RS3*2}
  L(s,\sym^2 f\times \phi) = \sum_{m,n\geq1} \frac{A(m,n)\lambda_\phi(n)}{(m^2n)^s}, \quad \Re(s)>1.
\end{equation}
It has an analytic continuation to the whole complex plane.
See e.g. \cite[\S5]{LuoSarnak2004} for these facts and further background.

We have the following estimate of the first moment of $L$-functions.

\begin{theorem}\label{thm:1moment}
  Let $K>2$ be  sufficiently large. 
  Let $\phi$ be an even Hecke--Maass cusp form for $\Gamma$ with the spectral parameter $t_\phi$.
   Assume $t_\phi\leq K^{1/2-\varepsilon}$.
  Then  we have
  \[
    \sum_{K<k\leq 2K}  \sum_{f\in H_k}  L(1/2,\sym^2 f\times\phi)
    \ll K^{2+\varepsilon} ,
  \]
  for any $\varepsilon>0$.
\end{theorem}

This extends a result of Luo--Sarnak \cite[\S5]{LuoSarnak2004}, who treated the case of fixed $t_\phi$.

Analogously, we will need the following estimate for the second moment of the symmetric square $L$-functions.

\begin{theorem}\label{thm:2moment}
  Let $K>2$ be  sufficiently large.   Let $t\in \mathbb{R}$ such that $|t|\leq K^{1/2-\varepsilon}$.
  Then  we have
  \[
    \sum_{K<k\leq 2K}
    \sum_{f\in H_k}  |L(1/2+it,\sym^2 f )|^2
    \ll K^{2+\varepsilon},
  \]
  for any $\varepsilon>0$.
\end{theorem}

We also require an upper bound for a mixed moment of $L$-functions. Interestingly, this can be deduced from $L^4$-norm bounds for Hecke eigenforms.
\begin{theorem}\label{thm:mixedmoment}
  Let $g\in H_\ell$. Then we have
  \[
    \sum_{t_{\phi} \ll \ell^{1/2+\varepsilon}}
   L(1/2,\phi) L(1/2,\sym^2g\times \phi)  \exp\left( - \frac{t_{\phi}^2}{\ell}\right)
   \ll \ell^{4/3+\varepsilon}.
  \]
\end{theorem}

\begin{remark}
  One may apply the Cauchy--Schwarz inequality and the spectral large sieve inequality to prove
  \[
    \sum_{t_{\phi} \ll \ell^{1/2+\varepsilon}} L(1/2,\phi) L(1/2,\sym^2g\times \phi)
   \ll \ell^{7/4+\varepsilon},
  \]
  which is weaker than Theorem \ref{thm:mixedmoment}.
  I learned this idea to improve the above bounds from Liangxun Li and Chengliang Guo (see \cite[Lemma 6.1]{Guo}).
\end{remark}

\subsection{Plan for this paper}
The rest of this paper is organized as follows.
In \S \ref{sec:preliminaries}, we give some lemmas on $L$-functions and sums of Fourier coefficients.
In \S \ref{sec:var}, we use the spectral method and Watson's formula to prove Theorem \ref{thm:var}.
In \S \ref{sec:moment}, we prove theorems on moments of $L$-functions.

\medskip
\textbf{Notation.}
Throughout the paper, $\varepsilon$ is an arbitrarily small positive number;
all of them may be different at each occurrence.
As usual, $e(x)=e^{2\pi i x}$.
We use $y\asymp Y$ to mean that $c_1 Y\leq |y|\leq c_2 Y$ for some positive constants $c_1$ and $c_2$.
The symbol $\ll_{a,b}$ denotes that the implied constant depends at most on $a$ and $b$.


\section{Preliminaries}\label{sec:preliminaries}

\subsection{$L$-functions}

Let $f\in H_k$ and $\phi$ an even Hecke--Maass cusp form for $\Gamma$.
The functional equation of the symmetric square $L$-function $L(s,\sym^2 f)$ is
\[
  \Lambda(s,\sym^2 f) = \Lambda(1-s,\sym^2 f),
\]
where the completed $L$-function is defined by
\[
  \Lambda(s,\sym^2 f) := L_\infty(s,\sym^2 f)  L(s,\sym^2 f),
\]
with
\[
  L_\infty(s,\sym^2 f) := \pi^{-3s/3} \Gamma\left(\frac{s+1}{2} \right)
  \Gamma\left(\frac{s+k-1}{2} \right)
  \Gamma\left(\frac{s+k}{2} \right) .
\]
The functional equation of the Rankin--Selberg $L$-function $L(s,\sym^2 f\times \phi)$ is
\[
  \Lambda(s,\sym^2 f\times \phi) = \Lambda(1-s,\sym^2 f\times \phi),
\]
where the completed $L$-function is defined by
\[
  \Lambda(s,\sym^2 f\times \phi) := L_\infty(s,\sym^2 f\times \phi)
  L(s,\sym^2 f\times \phi),
\]
with
\[
  L_\infty(s,\sym^2 f\times \phi) := \pi^{-3s}\prod_{\pm} \Gamma\left(\frac{s+1\pm it_\phi}{2} \right)
  \Gamma\left(\frac{s+k-1\pm it_\phi}{2} \right)
  \Gamma\left(\frac{s+k\pm it_\phi}{2} \right) .
\]
See e.g. \cite[\S1]{Khan} and \cite[\S5]{LuoSarnak2004}.
Note that $L(1/2,\sym^2 f \times \phi)\geq0$ (see \cite{Lapid}). We have the following approximate functional equations.

\begin{lemma} \label{lem:AFE}
  Let $t\in \mathbb{R}$.  We have
  \[
    L(1/2+it,\sym^2 f) =  \sum_{ n\geq1} \frac{\lambda_f(n^2)}{n^{1/2+it}} V_{3}^+(n;t)
    + \sum_{ n\geq1} \frac{\lambda_f(n^2)}{n^{1/2-it}} V_{3}^-(n;t),
  \]
  and 
  \[
    L(1/2,\sym^2 f\times \phi) = 2 \sum_{m,n\geq1} \frac{A(m,n)\lambda_\phi(n)}{(m^2n)^{1/2}} V_6(m^2n),
  \]
  where 
  \[
    V_{3}^+(y;t) := \frac{1}{2\pi i} \int_{(2)} \frac{L_\infty(1/2+it+s,\sym^2 f) }{L_\infty(1/2+it,\sym^2 f) } \zeta (1+2it+2s) y^{-s} e^{s^2} \frac{\dd s}{s}, 
  \]
  \[
    V_{3}^-(y;t)  := \frac{L_\infty(1/2-it,\sym^2 f) }{L_\infty(1/2+it,\sym^2 f) } V_3^+(y;-t),
  \]
  and 
  \[
    V_6(y) := \frac{1}{2\pi i} \int_{(2)} \frac{L_\infty(1/2+s,\sym^2 f\times \phi) }{L_\infty(1/2,\sym^2 f\times \phi) } y^{-s} e^{s^2} \frac{\dd s}{s}.
  \]
\end{lemma}

\begin{proof}
  This is  \cite[Theorem 5.3]{IwaniecKowalski2004analytic}.
\end{proof}

\begin{lemma}\label{lem:V}
  Let $K>1$ be a sufficiently large number, and $\eta\in(0,1/10)$ a small number. Assume $k\asymp K$, $-K^{1/2-\eta}\leq t\leq K^{1/2-\eta}$, and $t_\phi\leq K^{1/2-\eta}$. Then we have
  \[
    V_3^+(y;t) \ll \left( \frac{y}{ K (1+|t|)^{1/2} }  \right)^{-A},\quad 
    V_6(y) \ll \left( \frac{y}{ K^2 t_\phi }  \right)^{-A},
  \]
  for any $A>0$. Moreover, we have
  \[
    V_3^+(y;t) =  \frac{1}{2\pi i} \int_{\varepsilon-i K^{\varepsilon}}^{\varepsilon+i K^{\varepsilon}} \frac{L_\infty(1/2+it+s,\sym^2 f) }{L_\infty(1/2+it,\sym^2 f) } \zeta (1+2it+2s) y^{-s} e^{s^2} \frac{\dd s}{s} + O\left(K^{-2025}\right), 
  \]
  \[
    V_6(y) = \frac{1}{2\pi i} \int_{\varepsilon-i K^{\varepsilon}}^{\varepsilon+i K^{\varepsilon}}
    \left(\frac{k-1}{2}\right)^{2s} \sum_{j=0}^{J} \frac{P_j(s,t_\phi)}{(k-1)^{j}}
    \pi^{-3s} \prod_{\pm} \frac{\Gamma\left(\frac{s+3/2\pm it_\phi}{2} \right) }{\Gamma\left(\frac{3/2\pm it_\phi}{2} \right)} y^{-s} e^{s^2} \frac{\dd s}{s} + O\left(K^{-2025}\right),
  \]
  where $J=J(\eta)$ is a sufficiently large number, and $P_j(s,t_\phi)$ is a polynomial of degree $2j$.
\end{lemma}

\begin{proof}
  The first claim is standard. See e.g.  \cite[Proposition 5.4]{IwaniecKowalski2004analytic}. 
  The second claim follows from the Stirling's formula. 
\end{proof}

\subsection{The Petersson trace formula}

The Petersson trace formula is given by the following basic orthogonality relation on $H_k$.

\begin{lemma}\label{lem:PTF}
  Let $k\geq12$ be an even integer, and $m,n\geq1$. Then we have 
  \[
    \frac{12\zeta(2)}{(k-1)} \sum_{f\in H_k} \frac{\lambda_f(m)\lambda_f(n)}{ L(1,\operatorname{sym}^2 f)  }
 \;   = \; \delta_{m,n} + 2\pi i^{-k} \sum_{c=1}^{\infty}
                 \frac{S(m,n;c)}{c}J_{k-1}\left(\frac{4\pi\sqrt{mn}}{c}\right),
  \]
  where $\delta_{m,n}=1$ if $m=n$, $0$ otherwise, $S(m,n;c)=\sum_{\substack{d\mod c \\ (d,c)=1} } e(\frac{md+n\bar{d}}{c})$   the Kloosterman sum, and $J_{k-1}(v)$  the $J$-Bessel function. Here $\bar{d}$ is the inverse of $d$ modulo $c$. 
\end{lemma}
\begin{proof}
  See e.g.  \cite[Proposition 14.5]{IwaniecKowalski2004analytic} and
   \cite[\S2.1]{BlomerKhanYoung2013distribution}.
\end{proof}

\subsection{Fourier coefficients}


Let $\phi$ be a Hecke--Maass cusp form for $\SL_2(\mathbb{Z})$ with the spectral parameter $t_\phi$. 
We have the following strong bounds on the $\GL(2)$ exponential sums. 
\begin{lemma}\label{eqn:GL2exp-sum}
  For any $\alpha\in \mathbb{R}$, we have
  \[
    \sum_{n\leq N} \lambda_\phi(n) e(n\alpha) \ll_\varepsilon N^{1/2+\varepsilon} t_\phi^{1/2+\varepsilon},
  \]
  for any $\varepsilon>0$.
\end{lemma}

\begin{proof}
  This is  \cite[Theorem 1.2]{Godber2013additive}.
\end{proof}

\subsection{An average of the $J$-Bessel function}

We will use the following estimate of an average of the $J$-Bessel function.
\begin{lemma}\label{lem:J}
  For $x>0$, we have 
  \begin{multline*}
    \sum_{k\equiv 0 \; \mathrm{mod} \; 2} 2i^k W\left(\frac{k-1}{K}\right) J_{k-1}(x) 
    \\
    = 
    -\frac{K}{\sqrt{x}} \Im\left\{ e(-1/8) e^{ix} \breve{W}(K^2/2x)  \right\}
    + O\left(\frac{x}{K^4} \int_{\mathbb{R}} v^4 |\hat{W}(v)| \dd v \right),
  \end{multline*}
  where 
  $\breve{W}(v)=\int_{0}^{\infty} \frac{W(\sqrt{u})}{\sqrt{2\pi u}} e^{iuv} \dd u$ and 
  $\hat{W}(v)=\int_{\mathbb{R}} W(u) e(-uv)\dd u$. 
\end{lemma}

\begin{proof}
  This is \cite[Lemma 2.3]{Khan}.
\end{proof}

By integrating by parts several times we get that  
$\breve{W}(v)\ll (1+|v|)^{-B}$ 
and $\hat{W}(v)\ll (1+|v|)^{-B}$ for any $B \geq 0$.

\section{The variance}\label{sec:var}

In this section, we prove Theorem \ref{thm:var}.
For $f\in H_k$ and $g\in H_\ell$, by \cite[Eq. (3.1), (3.2), (3.12) and \S3.2.2]{Huang2024joint} we have
\begin{multline*}
  \frac{1}{\vol(\Gamma\backslash\mathbb{H})}\langle |F|^2,|G|^2\rangle - 1
  \\
  \ll
  \frac{1}{\sqrt{k\ell}} \sum_{t_\phi\ll \ell^{1/2+\varepsilon}}
  \frac{L(1/2,\phi)L(1/2,\sym^2f\times \phi)^{1/2}L(1/2,\sym^2g\times \phi)^{1/2}}
  {L(1,\sym^2f) L(1,\sym^2g) L(1,\sym^2 \phi)} \exp\left( - \frac{t_\phi^2}{2\ell}\right)
  \\
  +
  \frac{1}{\sqrt{k\ell}} \int_{|t|\leq \ell^{1/2+\varepsilon}}
  \frac{|\zeta(1/2+it)|^2 |L(1/2+it,\sym^2 f)L(1/2+it,\sym^2 g)|} {L(1,\sym^2 f)L(1,\sym^2 g)  |\zeta(1+2it)|^2 } \dd t
  + k^{-2025}.
\end{multline*}
Note that we have $L(1/2,\phi)\geq0$ and $L(1/2,\sym^2f\times \phi)\geq0$ (see \cite{KatokSarnak,Lapid}). 
Hence by the Cauchy--Schwarz inequality, we have
\begin{equation}\label{eqn:var2S}
  \frac{1}{K^2} \sum_{K<k\leq 2K}
    \sum_{f\in H_k}  \Big|\frac{1}{\vol(\Gamma\backslash\mathbb{H})}\langle |F_k|^2, |G_\ell|^2 \rangle-1 \Big|^2
    \ll
    \mathcal{S}_c + \mathcal{S}_e + K^{-2025},
\end{equation}
where
\begin{equation}\label{eqn:Sc}
  \mathcal{S}_c := \frac{1}{K^3 \ell} \sum_{K<k\leq 2K}
    \sum_{f\in H_k} \Big|\sum_{t_\phi\ll \ell^{1/2+\varepsilon}}
  \frac{L(1/2,\phi)L(1/2,\sym^2f\times \phi)^{1/2}L(1/2,\sym^2g\times \phi)^{1/2}}
  {L(1,\sym^2f) L(1,\sym^2g) L(1,\sym^2 \phi)}\Big|^2
\end{equation}
and
\begin{equation}\label{eqn:Se}
  \mathcal{S}_e:= \frac{1}{K^3 \ell} \sum_{K<k\leq 2K}
    \sum_{f\in H_k} \Big|
    \int_{|t|\leq \ell^{1/2+\varepsilon}}
  \frac{|\zeta(1/2+it)|^2 |L(1/2+it,\sym^2 f)L(1/2+it,\sym^2 g)|} {L(1,\sym^2 f)L(1,\sym^2 g)  |\zeta(1+2it)|^2 } \dd t  \Big|^2.
\end{equation}
We will estimate $\mathcal{S}_c$ and $\mathcal{S}_e$ separately.

\subsection{The Eisentein series contribution}

Note that by \cite{Iwaniec} and \cite{HL}, we have 
\begin{equation}\label{eqn:sym^2-1}
  \zeta(1+2it) = (1+|t|)^{o(1)}, \quad L(1,\sym^2 f)=k^{o(1)}, \quad  L(1,\sym^2 g) =\ell^{o(1)}.
\end{equation}
We have
\begin{equation*}
  \mathcal{S}_e \ll  \frac{K^\varepsilon}{K^3 \ell} \sum_{K<k\leq 2K}
    \sum_{f\in H_k} \Big|
    \int_{|t|\leq \ell^{1/2+\varepsilon}} |\zeta(1/2+it)|^2 |L(1/2+it,\sym^2 f)L(1/2+it,\sym^2 g)|  \dd t  \Big|^2.
\end{equation*}
By the Cauchy--Schwarz inequality, we have
\begin{multline*}
  \mathcal{S}_e \ll  \frac{K^\varepsilon}{K^3 \ell} \sum_{K<k\leq 2K}
    \sum_{f\in H_k}
    \int_{|t_1|\leq \ell^{1/2+\varepsilon}} |\zeta(1/2+it_1)|^4 |L(1/2+it_1,\sym^2 f)|^2  \dd t_1 \\
    \cdot \int_{|t_2|\leq \ell^{1/2+\varepsilon}}  |L(1/2+it_2,\sym^2 g)|^2  \dd t_2 \\
    =  \frac{K^\varepsilon}{K^3 \ell}
    \int_{|t_1|\leq \ell^{1/2+\varepsilon}} |\zeta(1/2+it_1)|^4
    \sum_{K<k\leq 2K} \sum_{f\in H_k}  |L(1/2+it_1,\sym^2 f)|^2  \dd t_1 \\
    \cdot \int_{|t_2|\leq \ell^{1/2+\varepsilon}}  |L(1/2+it_2,\sym^2 g)|^2  \dd t_2.
\end{multline*}

\begin{lemma}\label{lem:moment-sym2-t}
  Let $g\in H_\ell$. Then   we have
  \[
    \int_{|t|\leq \ell^{1/2+\varepsilon}}  |L(1/2+it,\sym^2 g)|^2  \dd t
    \ll \ell^{5/4+\varepsilon}.
  \]
\end{lemma}

\begin{proof}
  Note that for $g\in H_\ell$, the analytic conductor of $L(1/2+it,\sym^2 g)$ is $\ell^2 (3+|t|)$. By the approximate functional equation and the mean value estimate of Dirichlet polynomials (see \cite[Theorem 5.3 \& Theorem 9.1]{IwaniecKowalski2004analytic}), we get
  \[
    \int_{T<|t|\leq 2T}  |L(1/2+it,\sym^2 g)|^2  \dd t
    \ll \ell^{1+\varepsilon} T^{1/2},
  \]
  for $T\leq \ell^{1/2+\varepsilon}$. This completes the proof of the lemma. 
\end{proof}

Recall that we have the following well known bound
\[
  \int_{|t|\leq T} |\zeta(1/2+it)|^4 \dd t \ll T^{1+\varepsilon}.
\]
By Theorem \ref{thm:2moment} and Lemma \ref{lem:moment-sym2-t}, we have
\begin{equation}\label{eqn:Se<<}
   \mathcal{S}_e \ll  \frac{K^\varepsilon}{K^3 \ell} K^2 \ell^{1/2+5/4} \ll \ell^{3/4} K^{\varepsilon-1}.
\end{equation}

\subsection{The cusp form contribution}

By \eqref{eqn:Sc}, \eqref{eqn:sym^2-1}, and the Cauchy--Schwarz inequality, we have
\begin{multline*}
  \mathcal{S}_c  \ll
  \frac{K^\varepsilon}{K^3 \ell}
    \sum_{K<k\leq 2K} \sum_{f\in H_k} \sum_{t_\phi\ll \ell^{1/2+\varepsilon}}
   L(1/2,\phi)  L(1/2,\sym^2f\times \phi)      \\
  \cdot
  \sum_{t_{\phi'} \ll \ell^{1/2+\varepsilon}}
   L(1/2,\phi') L(1/2,\sym^2g\times \phi')  \exp\left( - \frac{t_{\phi'}^2}{\ell}\right)\\
   =
  \frac{K^\varepsilon}{K^3 \ell} \sum_{t_\phi\ll \ell^{1/2+\varepsilon}}
   L(1/2,\phi)
   \left( \sum_{K<k\leq 2K} \sum_{f\in H_k}  L(1/2,\sym^2f\times \phi) \right)    \\
  \cdot
  \sum_{t_{\phi'} \ll \ell^{1/2+\varepsilon}}
   L(1/2,\phi') L(1/2,\sym^2g\times \phi')  \exp\left( - \frac{t_{\phi'}^2}{\ell}\right)  .
\end{multline*}

Note that the spectral large sieve of Deshouillers and Iwaniec \cite[Theorem 2]{DI} gives
\[
  \sum_{t_\phi\ll \ell^{1/2+\varepsilon}}  L(1/2,\phi)
   \ll \bigg( \sum_{t_\phi\ll \ell^{1/2+\varepsilon}} 1 \bigg)^{1/2}
    \bigg( \sum_{t_\phi\ll \ell^{1/2+\varepsilon}} L(1/2,\phi)^2 \bigg)^{1/2}
   \ll \ell^{1+\varepsilon}.
\]
Together with Theorems \ref{thm:1moment} and \ref{thm:mixedmoment}, we have
\begin{equation}\label{eqn:Sc<<}
  \mathcal{S}_c
  \ll
  \frac{K^\varepsilon}{K^3 \ell} \sum_{t_\phi\ll \ell^{1/2+\varepsilon}}
   L(1/2,\phi)  K^{2+\varepsilon} \ell^{4/3+\varepsilon}
  \ll  K^{-1+\varepsilon} \ell^{4/3.}
\end{equation}
Combining \eqref{eqn:var2S}, \eqref{eqn:Se<<} and \eqref{eqn:Sc<<}, we complete the proof of Theorem \ref{thm:var}.

\subsection{An application to $L$-functions} \label{subsec:application}

Now we prove Corollary \ref{cor:moment8}. 
It suffices to show that 
\begin{equation}\label{eqn:<f2g2>=1}
  \frac{1}{\vol(\Gamma\backslash\mathbb{H})}\langle |F_k|^2, |G_\ell|^2 \rangle
  = 1 + O(K^{-\varepsilon}),
\end{equation}
for almost all even integer $k\in(K,2K]$ and almost all $f\in H_k$. 
Let 
\[ 
  \mathscr{S}_k := \sum_{f\in H_k}  \Big|\frac{1}{\vol(\Gamma\backslash\mathbb{H})}\langle |F_k|^2, |G_\ell|^2 \rangle-1 \Big|^2.
\]
For $\ell\leq K^{3/4-\varepsilon}$, Theorem \ref{thm:var} gives
$
   \sum_{K<k\leq2K} \mathscr{S}_k \ll K^{2-\varepsilon/3}. 
$
Hence  we have
\[
   \sum_{\substack{K<k\leq2K \\ \mathscr{S}_k\geq K^{1-\varepsilon/6}}} 1 
   \ll \sum_{\substack{K<k\leq2K  }} \mathscr{S}_k  / K^{1-\varepsilon/6}
   \ll K^{1-\varepsilon/6}. 
\]
So for all but $O(K^{1-\varepsilon/6})$ even integer $k\in [K,2K]$, we have $\mathscr{S}_k\leq K^{1-\varepsilon/6}$. 
For those $k$, we have 
\[
  \sum_{\substack{f\in H_k \\ |\frac{1}{\vol(\Gamma\backslash\mathbb{H})}\langle |F_k|^2, |G_\ell|^2 \rangle-1| \geq K^{-\varepsilon/24}}} 1
  \leq \sum_{\substack{f\in H_k}}  
  \Big|\frac{1}{\vol(\Gamma\backslash\mathbb{H})}\langle |F_k|^2, |G_\ell|^2 \rangle-1 \Big|^2 /K^{-\varepsilon/12} \ll K^{1-\varepsilon/12}.
\]
Hence for all but $O(K^{1-\varepsilon/12})$ forms $f\in H_k$, we have 
$|\frac{1}{\vol(\Gamma\backslash\mathbb{H})} \langle |F_k|^2, |G_\ell|^2 \rangle-1| \leq K^{-\varepsilon/24}$. This proves \eqref{eqn:<f2g2>=1}, 
and hence Corollary \ref{cor:moment8} by using \eqref{eqn:<f2g2>=L-values} and \eqref{eqn:sym^2-1}.

\section{Moments of $L$-functions}\label{sec:moment}

In this section, we prove Theorems \ref{thm:1moment}, \ref{thm:2moment}, and \ref{thm:mixedmoment}.

\subsection{A first moment of the Rankin-Selberg $L$-functions}
In this subsection, we will follow Luo--Sarnak's method in \cite[\S5]{LuoSarnak2004} to prove Theorem \ref{thm:1moment}.
Let $W\in C^\infty(\mathbb{R})$ such that $\supp W\subset [1/2,3]$ and $W^{(j)}(x)\ll_j 1$. 
It suffices to prove that for $t_\phi\leq K^{1/2-\varepsilon}$, we have  
\begin{equation}\label{eqn:M1<<K}
   \mathcal{M}_1 := \sum_{k\equiv 0 \; \mathrm{mod}\; 2} W\left(\frac{k-1}{K}\right)
    \frac{2\pi^2}{k-1}\sum_{f\in H_k} \frac{ L(1/2,\sym^2 f\times\phi)}{L(1,\sym^2 f)}
    \ll K^{1+\varepsilon}.
\end{equation}

\subsubsection{Applying the approximate functional equation}
By Lemma \ref{lem:AFE},  we have 
\[
   \mathcal{M}_1 = 2 \sum_{k\equiv 0 \; \mathrm{mod}\; 2} W\left(\frac{k-1}{K}\right)
    \frac{2\pi^2}{k-1}\sum_{f\in H_k} \frac{1}{L(1,\sym^2 f)}
    \sum_{m,n\geq1} \frac{A(m,n)\lambda_\phi(n)}{(m^2n)^{1/2}} V_6(m^2n).
\]
By Lemma \ref{lem:V} and a smooth partition of unity,  we arrive at  
\begin{multline}\label{eqn:M1}
   \mathcal{M}_1 = 2 \sum_{k\equiv 0 \; \mathrm{mod}\; 2} W\left(\frac{k-1}{K}\right)
    \frac{2\pi^2}{k-1}\sum_{f\in H_k} \frac{1}{L(1,\sym^2 f)} 
    \\ \cdot
    \sum_{m^2n \leq K^{2+\varepsilon}t_\phi} \frac{A(m,n)\lambda_\phi(n)}{(m^2n)^{1/2}} V_6(m^2n)
    + O(K^{-2025})
    \\ \ll  K^\varepsilon  \sup_{s=\varepsilon+i\tau, \ \tau\in[-K^\varepsilon,K^\varepsilon]} 
    \sup_{1\leq N \leq K^{2+\varepsilon}t_\phi}
   | \mathcal{M}_1(s,N) | + 1,
\end{multline}
where 
\[
  \mathcal{M}_1(s,N) := \sum_{k\equiv 0 \; \mathrm{mod}\; 2} W_s\left(\frac{k-1}{K}\right)
    \frac{2\pi^2}{k-1}\sum_{f\in H_k} \frac{1}{L(1,\sym^2 f)}
    \sum_{m^2n \geq1 } \frac{A(m,n)\lambda_\phi(n)}{(m^2n)^{1/2}} V_s\left(\frac{m^2n}{N}\right).
\]
Here $W_s(x)=W(x) \sum_{j=0}^{J}  P_j(s,t_\phi)  K^{2s-j}x^{2s-j}$ satisfying that 
\begin{equation}\label{eqn:WVs}
  \supp W_s\subset [1/2,3] \quad \textrm{ and } \quad  
  W_s^{(j)}(x) \ll_j K^{(j+2)\varepsilon}.
\end{equation}
We have similar properties for $V_s$.  
By \eqref{eqn:A(m,n)} and rearranging the sums, we get 
\begin{multline*}
  \mathcal{M}_1(s,N) = 
    \sum_{m^2n \geq 1} \frac{\lambda_\phi(n)}{(m^2n)^{1/2 }} V_s\left(\frac{m^2n}{N}\right) 
    \sum_{d\mid (m,n)} \mu(d) \\ \cdot
    \sum_{k\equiv 0 \; \mathrm{mod}\; 2} W_s\left(\frac{k-1}{K}\right)
    \frac{2\pi^2}{k-1}\sum_{f\in H_k} \frac{1}{L(1,\sym^2 f)}  A(m/d,1) A(1,n/d).
\end{multline*}
Making  changes of variables, we get  
\begin{multline*}
  \mathcal{M}_1(s,N) = 
    \sum_{d^3 m^2n \geq 1} \frac{\lambda_\phi(dn)}{(d^3m^2n)^{1/2 }}  \mu(d) V_s\left(\frac{d^3m^2n}{N}\right) 
    \sum_{m_1^2 m_2 =m} \sum_{n_1^2 n_2 =n} 
    \\ \cdot
    \sum_{k\equiv 0 \; \mathrm{mod}\; 2} W_s\left(\frac{k-1}{K}\right)
    \frac{2\pi^2}{k-1}\sum_{f\in H_k} \frac{\lambda_f(m_2^2) \lambda_f(n_2^2)}{L(1,\sym^2 f)} 
    \\
    = 
    \sum_{d^3 m_1^4 m_2^2 n_1^2 n_2 \geq 1} 
    \frac{\lambda_\phi(dn_1^2 n_2)}{(d^3m_1^4 m_2^2 n_1^2 n_2)^{1/2 }}  \mu(d)  
    V_s\left(\frac{d^3 m_1^4 m_2^2 n_1^2 n_2}{N}\right) 
    \\ \cdot
    \sum_{k\equiv 0 \; \mathrm{mod}\; 2} W_s\left(\frac{k-1}{K}\right)
    \frac{2\pi^2}{k-1}\sum_{f\in H_k} \frac{\lambda_f(m_2^2) \lambda_f(n_2^2)}{L(1,\sym^2 f)} .
\end{multline*}

\subsubsection{Applying the Petersson trace formula}
By Lemma \ref{lem:PTF}, we get 
\begin{equation}\label{eqn:M1=M10+M11}
  \mathcal{M}_1(s,N) 
    = 
    \mathcal{M}_{10}(s,N) + \mathcal{M}_{11}(s,N) ,
\end{equation}
where the diagonal contribution is 
\[
  \mathcal{M}_{10}(s,N)  := \sum_{d^3 m_1^4 m_2^2 n_1^2 n_2 \geq 1} 
    \frac{\lambda_\phi(dn_1^2 n_2)  \delta_{m_2,n_2} }{(d^3m_1^4 m_2^2 n_1^2 n_2)^{1/2 }}  \mu(d)  
    V_s\left(\frac{d^3 m_1^4 m_2^2 n_1^2 n_2}{N}\right) 
    \sum_{k\equiv 0 \; \mathrm{mod}\; 2} W_s\left(\frac{k-1}{K}\right) ,
\]
and the terms involving the $J$-Bessel function is
\begin{multline*}
  \mathcal{M}_{11}(s,N) 
    := 
    \sum_{d^3 m_1^4 m_2^2 n_1^2 n_2 \geq 1} 
    \frac{\lambda_\phi(dn_1^2 n_2)}{(d^3m_1^4 m_2^2 n_1^2 n_2)^{1/2+s}}  \mu(d) 
    V_s\left(\frac{d^3 m_1^4 m_2^2 n_1^2 n_2}{N}\right) 
    \\ \cdot
    \sum_{k\equiv 0 \; \mathrm{mod}\; 2} W_s\left(\frac{k-1}{K}\right)  2\pi i^{-k} \sum_{c=1}^{\infty}
                 \frac{S(m_2^2,n_2^2;c)}{c}J_{k-1}\left(\frac{4\pi m_2n_2}{c}\right) .
\end{multline*}

\subsubsection{The diagonal contribution}
We first deal with $\mathcal{M}_{10}(s,N)$. We have 
\[
  \mathcal{M}_{10}(s,N)  = \sum_{d^3 m_1^4   n_1^2 n_2^3  \geq 1} 
    \frac{\lambda_\phi(dn_1^2 n_2)}{(d^3m_1^4   n_1^2 n_2^3)^{1/2}}  \mu(d) 
    V_s\left(\frac{d^3 m_1^4  n_1^2 n_2^3}{N}\right)  
    \sum_{k\equiv 0 \; \mathrm{mod}\; 2} W_s\left(\frac{k-1}{K}\right) .
\]
By the Mellin inversion formula, we get 
\[
  \mathcal{M}_{10}(s,N)  = 
    \sum_{k\equiv 0 \; \mathrm{mod}\; 2} W_s\left(\frac{k-1}{K}\right) 
    \frac{1}{2\pi i} \int_{(2)} \sum_{d^3 m_1^4   n_1^2 n_2^3  \geq 1} 
    \frac{\lambda_\phi(dn_1^2 n_2) \mu(d) }{(d^3m_1^4   n_1^2 n_2^3)^{1/2+w}} 
    \tilde{V_s}(w) N^w \dd w  .
\]
  Note that 
  \[
    \sum_{d^3 m_1^4   n_1^2 n_2^3  \geq 1} 
    \frac{\lambda_\phi(dn_1^2 n_2) \mu(d) }{(d^3m_1^4   n_1^2 n_2^3)^{1/2+w}} 
    = \sum_{n \geq 1} \Big(  \sum_{d^3 m_1^4   n_1^2 n_2^3 =n} \lambda_\phi(dn_1^2 n_2) \mu(d) \Big) n^{-1/2-w}.
  \]
  Writing $dn_2=m$, we get 
  \[ 
    \sum_{d^3 m_1^4   n_1^2 n_2^3 =n} \lambda_\phi(dn_1^2 n_2) \mu(d)
    = \sum_{m^3 m_1^4   n_1^2 =n} \lambda_\phi(m n_1^2  ) \sum_{d\mid m}\mu(d)
    = \sum_{ m_1^4   n_1^2  =n} \lambda_\phi( n_1^2  ).
  \]
  For $\Re(w)>2$ we have 
  \[
    \sum_{d^3 m_1^4   n_1^2 n_2^3  \geq 1} 
    \frac{\lambda_\phi(dn_1^2 n_2) \mu(d) }{(d^3m_1^4   n_1^2 n_2^3)^{1/2+w}} 
    = L(1+2w,\sym^2\phi).
  \]
Hence we get 
\begin{equation}\label{eqn:M10<<}
  \mathcal{M}_{10}(s,N)  = 
    \sum_{k\equiv 0 \; \mathrm{mod}\; 2} W_s\left(\frac{k-1}{K}\right) 
    \frac{1}{2\pi i} \int_{(\varepsilon)} L(1+2w,\sym^2\phi)
    \tilde{V_s}(w) N^w \dd w \ll K^{1+\varepsilon} .
\end{equation}

\subsubsection{The Bessel function contribution}
Now we treat $\mathcal{M}_{11}(s,N)$. Rearranging the order of the sums, we get 
\begin{multline*}
  \mathcal{M}_{11}(s,N) 
    = \pi
    \sum_{d^3 m_1^4 m_2^2 n_1^2 n_2 \geq 1} 
    \frac{\lambda_\phi(dn_1^2 n_2)}{(d^3m_1^4 m_2^2 n_1^2 n_2)^{1/2+s}}  \mu(d) 
    V_s\left(\frac{d^3 m_1^4 m_2^2 n_1^2 n_2}{N}\right) 
    \\ \cdot \sum_{c=1}^{\infty}  \frac{S(m_2^2,n_2^2;c)}{c}
    \sum_{k\equiv 0 \; \mathrm{mod}\; 2} W_s\left(\frac{k-1}{K}\right)  
    2 i^{-k} J_{k-1}\left(\frac{4\pi m_2n_2}{c}\right) .
\end{multline*}
By Lemma \ref{lem:J}, we have 
\begin{equation}\label{eqn:M11<M111+M112}
  \mathcal{M}_{11}(s,N) 
    \ll |\mathcal{M}_{111}(s,N)|+|\mathcal{M}_{112}(s,N)|  + \mathcal{M}_{113}(N) ,
\end{equation}
where 
\begin{multline*}
  \mathcal{M}_{111}(s,N) 
    :=  
    \sum_{d^3 m_1^4 m_2^2 n_1^2 n_2 \geq 1} 
    \frac{\lambda_\phi(dn_1^2 n_2)}{(d^3m_1^4 m_2^2 n_1^2 n_2)^{1/2+s}}  \mu(d) 
    V_s\left(\frac{d^3 m_1^4 m_2^2 n_1^2 n_2}{N}\right) 
    \\ \cdot 
    \sum_{c=1}^{\infty}  \frac{S(m_2^2,n_2^2;c)}{c^{1/2}}
     \frac{K}{(m_2n_2)^{1/2}}   e\left(\frac{2 m_2n_2}{c}\right)  
     \breve{W}_s\left(\frac{K^2 c}{8\pi m_2n_2}\right)   , 
\end{multline*}
\begin{multline*}
  \mathcal{M}_{112}(s,N) 
    :=  
    \sum_{d^3 m_1^4 m_2^2 n_1^2 n_2 \geq 1} 
    \frac{\lambda_\phi(dn_1^2 n_2)}{(d^3m_1^4 m_2^2 n_1^2 n_2)^{1/2+s}}  \mu(d) 
    V_s\left(\frac{d^3 m_1^4 m_2^2 n_1^2 n_2}{N}\right) 
    \\ \cdot 
    \sum_{c=1}^{\infty}  \frac{S(m_2^2,n_2^2;c)}{c^{1/2}}
     \frac{K}{(m_2n_2)^{1/2}}   e\left(-\frac{2 m_2n_2}{c}\right) 
     \overline{\breve{W}_s\left(\frac{K^2 c}{8\pi m_2n_2}\right) }  , 
\end{multline*}
and 
\begin{equation*}
  \mathcal{M}_{113}(N) 
    :=  K^\varepsilon
    \sum_{d^3 m_1^4 m_2^2 n_1^2 n_2 \asymp N} 
    \frac{|\lambda_\phi(dn_1^2 n_2)|}{(d^3m_1^4 m_2^2 n_1^2 n_2)^{1/2}}    
    \sum_{c=1}^{\infty}  \frac{|S(m_2^2,n_2^2;c)|}{c} 
     \frac{1}{K^4} \frac{m_2n_2}{c} .
\end{equation*}

We first deal with $\mathcal{M}_{113}(N) $. By Weil's bound on the Kloosterman sums, we have 
\begin{equation*}
  \mathcal{M}_{113}(N) 
    \ll   K^{-4+\varepsilon}
    \sum_{d^3 m_1^4 m_2^2 n_1^2 n_2 \asymp N} 
    \frac{|\lambda_\phi(dn_1^2 n_2)|}{(d^3m_1^4  n_1^2 )^{1/2}}   n_2^{1/2}
    \sum_{c=1}^{\infty}  \frac{ (m_2^2,n_2^2,c)^{1/2}  }{c^{3/2-\varepsilon}} .
\end{equation*}
Using the bound $(m_2^2,n_2^2,c)\leq m_2^2$ and Kim--Sarnak's bounds $|\lambda_\phi(n)|\ll n^{7/64+\varepsilon}$, we get 
\begin{multline}\label{eqn:M113<<}
  \mathcal{M}_{113}(N) 
    \ll   K^{-4+\varepsilon}
    \sum_{d^3 m_1^4 m_2^2 n_1^2 n_2 \asymp N} 
    \frac{ (dn_1^2 n_2)^{7/64+\varepsilon}}{(d^3m_1^4  n_1^2 )^{1/2}}  m_2 n_2^{1/2}
    \sum_{c=1}^{\infty}  \frac{ 1}{c^{3/2-\varepsilon}}
    \\ 
    \ll N^{2-\varepsilon}  K^{-4+\varepsilon} \ll t_\phi^2
    \ll K ,
\end{multline}
provided by $N \leq K^{2+\varepsilon}t_\phi$ and $t_\phi\leq K^{1/2}$.

We now treat $\mathcal{M}_{111}(s,N) $. 
The estimate for $\mathcal{M}_{112}(s,N) $ will be the same. 
We have 
\begin{multline*}
  \mathcal{M}_{111}(s,N) 
    =  N^{-s} K^{-1} \sum_{c=1}^{\infty}  \frac{1}{c^{3/2}} 
    \sum_{d^3 m_1^4 m_2^2 n_1^2 \geq 1}  \mu(d)   \frac{1}{(d^3m_1^4 m_2 n_1^2 )^{1/2}}
    \\ \cdot 
    \sum_{n_2\geq1} \lambda_\phi(dn_1^2 n_2) 
    S(m_2^2,n_2^2;c)  e\left(\frac{2 m_2n_2}{c}\right)  
      \mathcal{V}\left(\frac{d^3 m_1^4 m_2^2 n_1^2 n_2}{N}\right) 
     \mathcal{W}\left(\frac{K^2 c}{ m_2n_2}\right)   , 
\end{multline*}
where $\mathcal{V}(y)=y^{-s}V_s(y)$ and $W_1(y)= y \breve{W}_s(y/8\pi)$. 
Note that  by \eqref{eqn:WVs} we have 
\begin{equation}\label{eqn:cV}
  \supp \mathcal{V}\subset [1/2,3] \quad \textrm{ and } \quad 
    \mathcal{V}^{(j)}(y) \ll_j  K^{j\varepsilon }. 
\end{equation}
By the definition of $\breve{W}_s$ and repeated integration by parts, we know 
\[
  \breve{W}_s^{(j)}(y) \ll_{j,A} \left(\frac{K^\varepsilon}{1+|y|}\right)^{A}, 
  \quad \textrm{for any $A\geq0$.}
\]
 Hence we have 
\begin{equation}\label{eqn:cW}
  \mathcal{W}^{(j)}(y) \ll_{j,A} K^\varepsilon \left(\frac{K^\varepsilon}{1+|y|}\right)^{A}, 
  \quad \textrm{for any $A\geq0$.}
\end{equation}
These show that $m_2n_2\leq 3N$ and the contribution from $K^2c/(m_2n_2)\leq K^{2\varepsilon}$ is negligibly small. So we can truncate the $c$-sum at $c\ll N K^{\varepsilon-2}$. 

By the Hecke relations, we get $\lambda_\phi(dn_1^2 n_2) =\sum_{a\mid (dn_1^2,n_2)}\mu(a)\lambda_\phi(dn_1^2/a) \lambda_\phi(n_2/a)$. Writing $n_2=an$, we obtain
\begin{multline*}
  \mathcal{M}_{111}(s,N) 
    =  N^{-s} K^{-1} \sum_{c\ll N K^{\varepsilon-2}}  \frac{1}{c^{3/2}} 
    \sum_{d^3 m_1^4 m_2^2 n_1^2 \geq 1}  \mu(d)   \frac{1}{(d^3m_1^4 m_2 n_1^2 )^{1/2}}
    \sum_{a\mid  dn_1^2}\mu(a)\lambda_\phi(dn_1^2/a) 
    \\ \cdot 
    \sum_{n\geq1} \lambda_\phi(n) 
    S(m_2^2,a^2 n^2;c)  e\left(\frac{2 m_2an}{c}\right)  
      \mathcal{V}\left(\frac{d^3 m_1^4 m_2^2 n_1^2 an}{N}\right) 
     \mathcal{W}\left(\frac{K^2 c}{ m_2an}\right)  + O(K^{-B}) , 
\end{multline*}
Breaking the $n$-sum into arithmetic progressions modulo $c$, we get 
\begin{multline*}
  \mathcal{M}_{111}(s,N) 
    =  N^{-s} K^{-1} \sum_{c\ll N K^{\varepsilon-2}}  \frac{1}{c^{3/2}} 
    \sum_{d^3 m_1^4 m_2^2 n_1^2 \geq 1}  \mu(d)   \frac{1}{(d^3m_1^4 m_2 n_1^2 )^{1/2}}
    \\ \cdot 
    \sum_{a\mid  dn_1^2}\mu(a)\lambda_\phi(dn_1^2/a) 
    \sum_{\alpha\mod c} S(m_2^2,a^2 \alpha^2;c)    
    e\left(\frac{2 m_2a \alpha}{c}\right)  
    \\ \cdot 
    \sum_{\substack{n\geq1 \\ n\equiv \alpha \mod c}} \lambda_\phi(n) 
      \mathcal{V}\left(\frac{d^3 m_1^4 m_2^2 n_1^2 an}{N}\right) 
     \mathcal{W}\left(\frac{K^2 c}{ m_2an}\right)  + O(K^{-B})   , 
\end{multline*}
By using the additive characters modulo $c$, we know that the innermost $n$-sum above is 
\begin{align*}
   S_1 & = \sum_{\substack{n\geq1 \\ n\equiv \alpha \mod c}} \lambda_\phi(n) 
      \mathcal{V}\left(\frac{d^3 m_1^4 m_2^2 n_1^2 an}{N}\right) 
     \mathcal{W}\left(\frac{K^2 c}{ m_2an}\right)  \\
   & = \frac{1}{c} \sum_{\beta\mod c} e( - \alpha \beta/c) 
   \sum_{ n\geq1  } \lambda_\phi(n) e(n\beta/c)
      \mathcal{V}\left(\frac{d^3 m_1^4 m_2^2 n_1^2 an}{N}\right) 
     \mathcal{W}\left(\frac{K^2 c}{ m_2an}\right) .
\end{align*}
By the partial summation formula we get 
\[
  S_1 \leq \max_{\beta\mod c} \bigg|
  \int_{\frac{N}{4d^3 m_1^4 m_2^2 n_1^2 a}}
  ^{\frac{4N}{d^3 m_1^4 m_2^2 n_1^2 a}} 
  \bigg(\sum_{ n\leq u } \lambda_\phi(n) e(n\beta/c)\bigg)
      \bigg(\mathcal{V}\left(\frac{d^3 m_1^4 m_2^2 n_1^2 a u}{N}\right) 
     \mathcal{W}\left(\frac{K^2 c}{ m_2a u}\right) \bigg)' \dd u
      \bigg|.
\]
By Lemma \ref{eqn:GL2exp-sum} we have
\[
  S_1 \ll 
  \int_{\frac{N}{4d^3 m_1^4 m_2^2 n_1^2 a}}
  ^{\frac{4N}{d^3 m_1^4 m_2^2 n_1^2 a}}  u^{1/2} t_\phi^{1/2+\varepsilon}
      \bigg(\mathcal{V}\left(\frac{d^3 m_1^4 m_2^2 n_1^2 a u}{N}\right) 
     \mathcal{W}\left(\frac{K^2 c}{ m_2a u}\right) \bigg)' \dd u .
\]
By \eqref{eqn:cV} and \eqref{eqn:cW} we get 
\[
  S_1 \ll K^\varepsilon \left(\frac{N}{d^3 m_1^4 m_2^2 n_1^2 a}\right)^{1/2 } t_\phi^{1/2}. 
\]
Hence 
\begin{multline*}
  \mathcal{M}_{111}(s,N) 
    \ll  K^{-1+\varepsilon} \sum_{c\ll N K^{\varepsilon-2}}  \frac{1}{c^{3/2}} 
    \sum_{d^3 m_1^4 m_2^2 n_1^2 \ll N}    \frac{1}{(d^3m_1^4 m_2 n_1^2 )^{1/2}}
    \sum_{a\mid  dn_1^2} \left(\frac{dn_1^2}{a}\right)^{1/2} 
    \\ \cdot 
    \sum_{\alpha\mod c} |S(m_2^2,a^2 \alpha^2;c)| \left(\frac{N}{d^3 m_1^4 m_2^2 n_1^2 a}\right)^{1/2 } t_\phi^{1/2} . 
\end{multline*}
By Weil's bound on the Kloosterman sums, we get 
\begin{align}
  \mathcal{M}_{111}(s,N)  
  & \ll  \frac{  N^{1/2}  t_\phi^{1/2} }{K^{1-\varepsilon}} 
    \sum_{c\ll N K^{\varepsilon-2}}  
    \sum_{  m_2   \ll N^{1/2}}    \frac{1}{m_2^{3/2}}   (m_2^2,c)^{1/2} \nonumber \\
  & \leq   \frac{  N^{1/2}  t_\phi^{1/2} }{K^{1-\varepsilon}} 
    \sum_{  m_2   \ll N^{1/2}}    \frac{1}{m_2^{3/2}} 
    \sum_{c\ll N K^{\varepsilon-2}}    \sum_{d\mid (m_2^2,c)} d^{1/2} \nonumber \\
  & \leq \frac{  N^{1/2}  t_\phi^{1/2} }{K^{1-\varepsilon}} 
    \sum_{  m_2   \ll N^{1/2}}    \frac{1}{m_2^{3/2}} 
       \sum_{d\mid m_2^2} d^{-1/2}  N K^{\varepsilon-2} 
    \ll \frac{ N^{3/2}t_\phi^{1/2}}{ K^{3-\varepsilon}}  \ll K^\varepsilon t_\phi^2. \label{eqn:M111<<}
\end{align}
In the last inequality, we have used the condition $N \leq K^{2+\varepsilon}t_\phi$. 

By \eqref{eqn:M1}, \eqref{eqn:M1=M10+M11}, \eqref{eqn:M10<<}, \eqref{eqn:M11<M111+M112}, 
\eqref{eqn:M113<<}, \eqref{eqn:M111<<}, we prove \eqref{eqn:M1<<K}. Hence we complete the proof of Theorem \ref{thm:1moment}. 

\subsection{A second moment of the symmetric square $L$-functions}
In this subsection, we prove Theorem \ref{thm:2moment}. \emph{Cf.} Khan \cite{Khan}. 
Let $W\in C^\infty(\mathbb{R})$ such that $\supp W\subset [1/2,3]$ and $W^{(j)}(x)\ll_j 1$. 
It suffices to prove that for $-K^{1/2-\varepsilon}\leq t\leq K^{1/2-\varepsilon}$, we have  
\[
   \mathcal{M}_2 := \sum_{k\geq12} W\left(\frac{k-1}{K}\right)
    \sum_{f\in H_k} \frac{| L(1/2+it,\sym^2 f)|^2 }{L(1,\sym^2 f)}
    \ll K^{2+\varepsilon},
\]
for any $\varepsilon>0$. 
  

\subsubsection{Applying the approximate functional equation}
By Lemmas \ref{lem:AFE} and \ref{lem:V}, and a smooth partition of unity, we get 
\[
  \mathcal{M}_2 \ll   \sup_{N\leq K^{1+\varepsilon}\sqrt{T}}
   \sum_{k\geq12} W\left(\frac{k-1}{K}\right) \sum_{f\in H_k} \frac{1 }{L(1,\sym^2 f)} 
   \left| \sum_{n\geq1} \frac{\lambda_f(n^2)}{n^{1/2+it}} V_3^+\left(n;t\right)  V_1\left(\frac{n}{N}\right) \right|^2 + 1,
\]
where $V_1(\xi)\in \mathcal{C}_c^\infty(\mathbb{R})$ with $\supp V_1\subseteq [1,2]$, $V_1^{(j)}(\xi)\ll_j 1$, for any $j\in \mathbb{Z}_{\geq0}$.  
Here we write $T=1+|t|$. 
Hence by Lemma \ref{lem:V} again and Stirling's formula, we obtain
\begin{multline*}
  \mathcal{M}_2 \ll K^{\varepsilon} \sup_{N\leq K^{1+\varepsilon}\sqrt{T}}
   \sideset{}{^{\rm even}} \sum_{k\geq12} W\left(\frac{k-1}{K}\right) \sum_{f\in H_k} \frac{1 }{L(1,\sym^2 f)}  
   \\
   \cdot \int_{\varepsilon-i K^{\varepsilon}}^{\varepsilon+i K^{\varepsilon}}  \left| \sum_{n\geq1} \frac{\lambda_f(n^2)}{n^{1/2+it}} \left(\frac{n}{N}\right)^{-s}  V_1\left(\frac{n}{N}\right) \right|^2 \dd s + 1.
\end{multline*}
Hence we have 
\begin{equation}\label{eqn:M2<<}
  \mathcal{M}_2 \ll K^{1+\varepsilon} \sup_{N\leq K^{1+\varepsilon}\sqrt{T}}
  \mathcal{M}_2(N) + 1,
\end{equation}
where 
\[
  \mathcal{M}_2(N):=
  \sideset{}{^{\rm even}} \sum_{k\geq12} W\left(\frac{k-1}{K}\right)
    \frac{12\zeta(2)}{(k-1)}\sum_{f\in H_k} \frac{1 }{L(1,\sym^2 f)} 
    \bigg| \sum_{n\geq1} \frac{\lambda_f(n^2)}{n^{1/2+it}} V\left(\frac{n}{N}\right) \bigg|^2,
\]
for certain $V(\xi)\in \mathcal{C}_c^\infty(\mathbb{R})$ with $\supp V\subseteq [1,2]$, $V^{(j)}(\xi)\ll_j K^{j\varepsilon}$, for any $j\in \mathbb{Z}_{\geq0}$.  
Opening the square and rearranging the sums, we have 
\begin{multline*}
  \mathcal{M}_2(N)= \sum_{m\geq1} \frac{1}{m^{1/2+it}} V\left(\frac{m}{N}\right) \sum_{n\geq1} \frac{1}{n^{1/2-it}} \overline{V\left(\frac{n}{N}\right)} \\
  \cdot
  \sideset{}{^{\rm even}}\sum_{k\geq12} W\left(\frac{k-1}{K}\right)
    \frac{12\zeta(2)}{(k-1)}\sum_{f\in H_k} \frac{\lambda_f(m^2)\lambda_f(n^2)}{L(1,\sym^2 f)}.
\end{multline*}

\subsubsection{Applying the Petersson trace formula}
By Lemma \ref{lem:PTF}, the second line of the above equation is equal to 
\[
  \sideset{}{^{\rm even}}\sum_{k\geq12} W\left(\frac{k-1}{K}\right)  \delta_{m,n} 
  + \sideset{}{^{\rm even}}\sum_{k\geq12} W\left(\frac{k-1}{K}\right) 2\pi i^{-k} \sum_{c=1}^{\infty}
                 \frac{S(m^2,n^2;c)}{c}J_{k-1}\left(\frac{4\pi mn}{c}\right) .
\]
Hence we have 
\begin{equation}\label{eqn:M2=M20+M21}
  \mathcal{M}_2(N)=\mathcal{M}_{20}(N)+\mathcal{M}_{21}(N),
\end{equation}
where the diagonal contribution  is 
\[
  \mathcal{M}_{20}(N) :=   \sum_{n\geq1} \frac{1}{n} \left|V\left(\frac{n}{N}\right)\right|^2 
  \sideset{}{^{\rm even}}\sum_{k\geq12} W\left(\frac{k-1}{K}\right)  
  \ll K.
\]
and  the terms involving the $J$-Bessel function is 
\begin{multline*}
  \mathcal{M}_{21}(N) := \pi \sum_{m\geq1} \frac{1}{m^{1/2+it}} V\left(\frac{m}{N}\right) \sum_{n\geq1} \frac{1}{n^{1/2-it}} \overline{V\left(\frac{n}{N}\right)}
  \\ \cdot 
  \sum_{c=1}^{\infty}   \frac{S(m^2,n^2;c)}{c}
  \sideset{}{^{\rm even}}\sum_{k\geq12} W\left(\frac{k-1}{K}\right) 2 i^{-k} J_{k-1}\left(\frac{4\pi mn}{c}\right).
\end{multline*}

\subsubsection{The Bessel function contribution} 
By Lemma \ref{lem:J}, we have 
\begin{multline*}
  \mathcal{M}_{21}(N) = \pi \sum_{m\geq1} \frac{1}{m^{1/2+it}} V\left(\frac{m}{N}\right) \sum_{n\geq1} \frac{1}{n^{1/2-it}} \overline{V\left(\frac{n}{N}\right)}  \sum_{c=1}^{\infty}   \frac{S(m^2,n^2;c)}{c}
  \\ \cdot   
  \left(-\frac{K}{\sqrt{x}} \Im\left\{ e(-1/8) e^{ix} \breve{W}(K^2/2x)  \right\}
    + O\left(\frac{x}{K^4} \int_{\mathbb{R}} v^4 |\hat{W}(v)| \dd v \right) \right),
\end{multline*}
where $x=4\pi mn/c$. 
The contribution from the error term is bounded by 
\[
  O\left( \sum_{m\ll N} \frac{1}{m^{1/2}}   \sum_{n\ll N} \frac{1}{n^{1/2}}   \sum_{c=1}^{\infty}   \frac{(m^2,n^2,c)^{1/2} c^{1/2+\varepsilon}}{c} \frac{mn}{cK^4} \right)
  = O\left( \frac{N^2}{K^4}   \sum_{n\ll N}    \sum_{c=1}^{\infty}   \frac{(n^2,c)^{1/2}}{c^{3/2-\varepsilon}}  \right).
\]
Note that $\sum_{n\ll N}    \sum_{c=1}^{\infty}   \frac{(n^2,c)^{1/2}}{c^{3/2-\varepsilon} }
\ll \sum_{n\ll N}  \sum_{d\mid n^2}  \sum_{\substack{c\geq1,d\mid c}}    \frac{d^{1/2}}{c^{3/2-\varepsilon}} \ll N$. 
The above is 
\[
  O(N^3 K^{-4}) = O(T^{3/2}K^{-1+\varepsilon}), 
\]
which is $O(K)$ if $T\leq K$. Hence we get 
\begin{equation}\label{eqn:M2<<M211}
  \mathcal{M}_2(N) \ll \mathcal{M}_{211}(N) + K, 
\end{equation}
where $\mathcal{M}_{211}(N)$ is defined by 
\[
  K \sum_{c=1}^{\infty}   \frac{1}{c^{1/2}} \sum_{m\geq1} \frac{1}{m^{1+it}} V\left(\frac{m}{N}\right) 
  \sum_{n\geq1} \frac{1}{n^{1-it}} \overline{V\left(\frac{n}{N}\right)}  S(m^2,n^2;c)  e\left(\pm \frac{2mn}{c}\right) \breve{W}\left(\frac{K^2c}{8\pi mn}\right) .
\]
Note that by  $\breve{W}(v)\ll (1+|v|)^B$, the contribution from terms with $N^2/c \leq K^{2-\varepsilon}$ is negligibly small. 
Hence we can truncate the $c$-sum at $c\leq N^2 /K^{2-\varepsilon}$, getting
\begin{multline}\label{eqn:M211<<}
  \mathcal{M}_{211}(N)
  = K \sum_{c\leq N^2 /K^{2-\varepsilon}}   \frac{1}{c^{1/2}} \sum_{m\geq1} \frac{1}{m^{1+it}} V\left(\frac{m}{N}\right) 
  \\
  \cdot\sum_{n\geq1} \frac{1}{n^{1-it}} \overline{V\left(\frac{n}{N}\right)}  S(m^2,n^2;c)  e\left(\pm \frac{2mn}{c}\right) \breve{W}\left(\frac{K^2c}{8\pi mn}\right)
  +  O_B(K^{-B}),
\end{multline}
for any $B>0$.

If $|t|\leq K^\varepsilon$, then we have 
\[
  \mathcal{M}_{211}(N) \ll K \sum_{c\ll K^\varepsilon} c \sum_{m\asymp N}\frac{1}{m}\sum_{n\asymp N}\frac{1}{n} 
  \ll K^{1+\varepsilon}. 
\]

If $K^\varepsilon \leq |t| \leq K^{1/2-\varepsilon}$, then we consider the $n$-sum in $\mathcal{M}_{211}(N)$, 
\begin{align*}
  \mathcal{S} & = \sum_{n\geq1}  S(m^2,n^2;c)  e\left(\pm \frac{2mn}{c}\right)  \frac{1}{n^{1-it}} \overline{V\left(\frac{n}{N}\right)} \breve{W}\left(\frac{K^2c}{8\pi mn}\right) \\
  & = \sum_{b\mod c} S(m^2,b^2;c)  e\left(\pm \frac{2mb}{c}\right)  \sum_{\substack{n\geq1\\n\equiv b\mod c}}   \frac{1}{n^{1-it}} \overline{V\left(\frac{n}{N}\right)} \breve{W}\left(\frac{K^2c}{8\pi mn}\right). 
\end{align*}
By the Poisson summation formula we get 
\[
  \mathcal{S} = \sum_{b\mod c} S(m^2,b^2;c)  e\left(\pm \frac{2mb}{c}\right) 
  \frac{1}{c} \sum_{n\in \mathbb{Z}}  e\left(\frac{nb}{c}\right)  
  \mathcal{I}(n),
\]
where 
\[
  \mathcal{I}(n) := \int_{\mathbb{R}} \frac{1}{y^{1-it}} \overline{V\left(\frac{y}{N}\right)} \breve{W}\left(\frac{K^2c}{8\pi my}\right) 
  e\left(-\frac{ny}{c}\right)  \dd y.
\]
By making a change of variable $y=N\xi$, we have 
\[
  \mathcal{I}(n) =  N^{it} \int_{\mathbb{R}} \frac{1}{\xi} \overline{V\left(\xi\right)} \breve{W}\left(\frac{K^2c}{8\pi mN\xi}\right)   e\left(\frac{t}{2\pi}\log \xi - \frac{nN}{c}\xi\right)  \dd \xi.
\]
By repeated integration by parts and the assumption $|t|\geq K^\varepsilon$, we have 
\[\mathcal{I}(0) \ll_B K^{-B},\]
for any $B>0$.
Recall that $N\leq K^{1+\varepsilon}\sqrt{T}$. For $|n|\geq1$ and $c\leq N^2/K^{2-\varepsilon}$, we have 
$|nN/c|\geq N/c\geq K^{2-\varepsilon}/N \geq K^{1-2\varepsilon}/\sqrt{T}$. 
If $T\ll K^{1/2}$, we have $|nN/c|\gg K^\varepsilon T$. 
Therefore $\frac{\dd}{\dd \xi}(\frac{t}{2\pi}\log \xi - \frac{nN}{c}\xi) \gg K^\varepsilon T$.
By repeated integration by parts, we have 
\[\mathcal{I}(n) \ll_B n^{-6} K^{-B},\]
for any $B>0$.
Hence by \eqref{eqn:M211<<} we have $\mathcal{M}_{211}(N) \ll_B K^{-B}$, for any $B>0$.

Combining \eqref{eqn:M2<<}, \eqref{eqn:M2=M20+M21}, and \eqref{eqn:M2<<M211}, we complete the proof of Theorem \ref{thm:2moment}.

\subsection{A mixed moment of $L$-functions}
In this subsection, we prove Theorem \ref{thm:mixedmoment}. 
By the discussion preceding \cite[Eq.~(3.7)]{Huang2024joint}, we know that for $|t|\leq \ell^{2/3}$, 
\[
  \frac{|\Gamma(\ell-1/2+it)|}{\Gamma(\ell)} 
  \asymp \frac{1}{\ell^{1/2}} \exp\left(-\frac{t^2}{2\ell}\right). 
\]
By the standard Rankin--Selberg method and the Watson formula \cite{Watson}, we have 
\begin{multline*}
  \|g\|_4^4 \asymp 1 + \frac{1}{\ell} \;\ \  \sideset{}{^{\rm even}}\sum_{t_\phi\leq \ell^{1/2+\varepsilon}}
  \frac{L(1/2,\phi) L(1/2,\sym^2 g\times \phi)}{L(1,\sym^2 g)^2 L(1,\sym^2 \phi)} \exp\left(-\frac{t_\phi^2}{\ell}\right) \\
  + \frac{1}{\ell} \int_{|t|\leq \ell^{1/2+\varepsilon}}
  \frac{|\zeta(1/2+it)|^2  |L(1/2+it,\sym^2 g)|^2} { L(1,\sym^2 g)^2  |\zeta(1+2it)|^2 } \exp\left(-\frac{t^2}{\ell}\right) \dd t .
\end{multline*}
Hence by Blomer--Khan--Young's $L^4$-norm bound \eqref{eqn:L^4} we have 
\begin{equation*}
   \sideset{}{^{\rm even}} \sum_{t_\phi\leq \ell^{1/2+\varepsilon}}
  \frac{L(1/2,\phi) L(1/2,\sym^2 g\times \phi)}{L(1,\sym^2 g)^2 L(1,\sym^2 \phi)} \exp\left(-\frac{t_\phi^2}{\ell}\right) \\
  \ll  \ell \|g\|_4^4  \ll \ell^{4/3+\varepsilon} .
\end{equation*}
By  \eqref{eqn:sym^2-1}, we complete the proof of Theorem \ref{thm:mixedmoment}.

\section*{Acknowledgements}
 
The author wants to thank Professors Jianya Liu and Ze\'ev Rudnick for encouragements.
He is grateful to the referee for his/her very helpful comments and suggestions.



\begin{thebibliography}{10}


\bibitem{BlomerKhanYoung2013distribution}
V.~Blomer, R.~Khan, and M.~Young,
\newblock Distribution of mass of holomorphic cusp forms.
\newblock {\em Duke Math. J.} 162 (2013), no. 14, 2609--2644.


















\bibitem{DI}
J.-M. Deshouillers and H. Iwaniec,
Kloosterman sums and Fourier coefficients of cusp forms.
\emph{Invent. Math.} 70 (1982/83), no. 2, 219--288.



\bibitem{Godber2013additive}
D. Godber,
Additive twists of Fourier coefficients of modular forms.
\emph{J. Number Theory} 133 (2013), no. 1, 83--104.

\bibitem{Guo}
C. Guo,
Joint cubic moment of Eisenstein series and Hecke--Maass cusp forms.
\emph{J. Number Theory} 276 (2025), 162--197.


\bibitem{HL}
J. Hoffstein and P. Lockhart, 
Coefficients of Maass forms and the Siegel zero.
With an appendix by Dorian Goldfeld, Hoffstein and Daniel Lieman. 
\emph{Ann. of Math. (2)} 140 (1994), no. 1, 161--181.

\bibitem{HS}
R. Holowinsky and K. Soundararajan,
Mass equidistribution for Hecke eigenforms.
\emph{Ann. of Math. (2)} 172 (2010), no. 2, 1517--1528.




\bibitem{Huang2024joint}
B. Huang,
Joint distribution of Hecke eigenforms.
To appear in \emph{Int. Math. Res. Not. IMRN}, 2026+.
\url{https://arxiv.org/abs/2406.03073}




\bibitem{HLWY}
B. Huang, S. Lester, I. Wigman, N. Yesha, 
On the supremum of random cusp forms. 
\emph{ArXiv Preprint} (2025), 23 pp.
\url{https://arxiv.org/abs/2508.16813}



\bibitem{Iwaniec}
H.~Iwaniec, 
Small eigenvalues of Laplacian for $\Gamma_0(N)$.
\emph{Acta Arith.} 56 (1990), no. 1, 65--82.




\bibitem{IwaniecKowalski2004analytic}
H.~Iwaniec and E.~Kowalski.
\newblock {\em Analytic number theory}, volume~53 of {\em American Mathematical
  Society Colloquium Publications}.
\newblock American Mathematical Society, Providence, RI, 2004.

\bibitem{KatokSarnak}
S. Katok and P. Sarnak,
Heegner points, cycles and Maass forms.
\emph{Israel J. Math.} 84 (1993), no. 1-2, 193--227.

\bibitem{Khan}
R. Khan, 
Non-vanishing of the symmetric square $L$-function at the central point. 
\emph{Proc. Lond. Math. Soc. (3)} 100 (2010), no. 3, 736--762.


\bibitem{Khan2014}
R. Khan, 
On the fourth moment of holomorphic Hecke cusp forms. 
\emph{Ramanujan J.} 34 (2014), no. 1, 83--107.











\bibitem{Lapid}
E. M. Lapid, On the nonnegativity of Rankin-Selberg $L$-functions at the center of symmetry.
\emph{Int. Math. Res. Not. IMRN} 2 (2003), 65--75.













%





\bibitem{LuoSarnak2004}
W. Luo and P. Sarnak,
Quantum variance for Hecke eigenforms.
\emph{Ann. Sci. \'Ecole Norm. Sup.} (4) 37 (2004), no. 5, 769--799.




\bibitem{RS}
Z. Rudnick and P. Sarnak,
The behaviour of eigenstates of arithmetic hyperbolic manifolds.
\emph{Comm. Math. Phys.} 161 (1994), no. 1, 195--213.








\bibitem{Watson}
T. Watson,
Rankin Triple Products and Quantum Chaos.
\emph{ArXiv preprint} (2008), 66 pp.
\url{https://arxiv.org/abs/0810.0425v3}

\bibitem{Xia}
H. Xia, On $L^\infty$ norms of holomorphic cusp forms.
\emph{J. Number Theory} 124 (2007), no. 2, 325--327.



\bibitem{Zenz}
P. Zenz,
Sharp bound for the fourth moment of holomorphic Hecke cusp forms.
\emph{Int. Math. Res. Not. IMRN} 2023, no. 16, 13562--13600.

\end{thebibliography}
\end{document}